\documentclass[12pt, twoside, reqno] {amsart}

\usepackage{amsfonts}
\usepackage{amssymb}
\usepackage{latexsym}
\usepackage{epsf,graphicx}
\usepackage{epsf,graphicx,latexsym,%
}

\newtheorem{example}{Example}[section]

\newcommand{\be} {\begin{eqnarray}}
\newcommand{\ee} {\end{eqnarray}}
\newcommand{\bep} {\begin{eqnarray*}}
\newcommand{\eep} {\end{eqnarray*}}

\newcommand {\Hol}{\mathop{\rm Hol}\nolimits}

\newcommand {\Id}{\mathop{\rm Id}\nolimits}

\renewcommand {\Re}{\mathop{\rm Re}\nolimits}

\newcommand {\DD}{\mathcal{D}}
\newcommand {\Ff}{\mathcal{F}}
\newcommand {\A}{\mathcal{A}}

\newcommand{\R}{{\mathbb R}}
\newcommand{\N}{{\mathbb N}}

\newcommand{\C}{{\mathbb C}}

\newtheorem{remar}{Remark}[section]
\newtheorem{examp}{Example}[section]
\newtheorem{defin}{Definition}[section]
\newtheorem{corol}{Corollary}[section]
\newtheorem{propo}{Proposition}[section]
\newtheorem{theorem}{Theorem}[section]
\newtheorem{lemma}{Lemma}[section]

\newcommand{\rema}{\begin{remar}\rm}
\newcommand{\erema}{$\blacktriangleright$\end{remar}}

\newcommand{\exa}{\begin{examp}\rm}
\newcommand{\eexa}{$\blacktriangleright$\end{examp}}

\def\lwvec(#1 #2){\linewd 0.1
           \lvec(#1 #2)
           \linewd 0.05}

\begin{document}

\title[Linearization of semicocycles]{Linearization of holomorphic semicocycles in Banach spaces}

\author[M. Elin]{Mark Elin}

\address{Department of Mathematics,
         Ort Braude College,
         Karmiel 21982,
         Israel}

\email{mark$\_$elin@braude.ac.il}

\author[F. Jacobzon]{Fiana Jacobzon}

\address{Department of Mathematics,
         Ort Braude College,
         Karmiel 21982,
         Israel}

\email{fiana@braude.ac.il}

\author[G. Katriel]{Guy Katriel}

\address{Department of Mathematics,
         Ort Braude College,
         Karmiel 21982,
         Israel}

\email{katriel@braude.ac.il}

\begin{abstract}

We consider
holomorphic semicocycles on the open unit ball in a Banach space taking values in a Banach algebra (studied previously in \cite{EJK, EJK-multi}).
We establish criteria for a semicocycle to be linearizable, that is, cohomologically equivalent to one independent of the spatial variable.

\vspace{2mm}
{\footnotesize Key words and phrases: holomorphic mapping, generated semigroup, Cauchy problem, linearization problem.

}

\end{abstract}
\maketitle

\centerline\today

\section{ Introduction}\label{sect-intro}

In the theory of autonomous dynamical systems, the basic equivalence relation is conjugacy: two semigroups which are conjugate
describe the same dynamics, up to a re-parametrization of the space acted upon. Thus, it is important to classify semigroups up to conjugacy. In particular, when dealing with smooth dynamical systems, one is
interested in the question of {\it{linearization}}, that is, determining whether a semigroup is conjugate to a
linear one. This is motivated by the fact that autonomous linear dynamical systems are the most well-understood type of dynamical systems.

For (semi)cocycles over semigroups an analogous equivalence relation is that of {\it{cohomological equivalence}}, and
the corresponding class of simplest semicocycles is that of {\it{constant}} (that is, independent of the spatial variable) semicocycles. Therefore in this context the problem of
`linearization' is that of determining whether a given semicocycle is cohomologous to a constant one.

In this paper we study semicocycles over semigroups of holomorphic self-mappings of a domain in a complex Banach space (the books \cite{R-S1, E-R-S-19} and references therein can be used as  good sources for the state of the art on semigroups of holomorphic mappings).
While, under an assumption of uniform joint continuity (see \cite{EJK-multi})  such semicocycles can be defined as solutions of certain nonautonomous differential equations which cannot
be explicitly solved, if a semicocycle is cohomolgous to a constant one it has an explicit representation.
Our central question is:  {\it Which semicocycles are linearizable --- that is cohomolgous to constant semicocycles? }

 In the case where $X=\C$, the question of linearizability was studied in \cite{EJK}, where we obtained a condition which is sharp in a generic sense.

 In multi-dimensional settings, the situation is inherently more complicated, as is also the case with respect to
 the question of conjugacy of semigroups  (see, for example, \cite{B-E-S,E-R-S-04,R-S1}).
   We obtain some simple sufficient conditions for such a linearization to exist, in terms of the semicocycle generator. The methods we use are quite different from those  used in \cite{EJK}, which relied on power-series expansions.

  The next section is devoted to some preliminary results and to definitions of the objects of our study, while in Section~\ref{sect-lineari} we introduce the linearizaton problem for semicocycles and discuss its equivalence to the linearization of skew-product semiflows. Section~\ref{Section-linearization} contains the main results of the paper and provides some conditions entailing that a semicocycle  over a semigroup with an interior attractive fixed  point is linearizable.

\section{Semigroups and semicocycles}\label{Pre}
\setcounter{equation}{0}

In this section we introduce the main objects studied in the paper and present some preliminary results.

Let $X$ and $Y$ be two complex Banach spaces endowed with the norms $\|\cdot\|_X$ and $\|\cdot\|_Y$, respectively. Let $\DD\subset X$ and $\Omega\subset Y$ be domains (connected open sets). Recall that a mapping $F:\mathcal{D}\to \Omega$ is said to be holomorphic if it is Fr\'ech{e}t differentiable at each point $x\in \DD$. By $\Hol(\DD,\Omega)$ we denote the set of all holomorphic mappings on $\mathcal{D}$ with values in $\Omega \subset Y$. Also we denote $\Hol(\DD):=\Hol(\DD,\DD)$, the set of all holomorphic self-mappings of $\mathcal{D}$.

A bounded subset $\DD^* \subset\DD$ is said to lie strictly inside $\DD$ if it is bounded away from the boundary $\partial\DD$, that is,  $\displaystyle\inf_{x\in\mathcal{D}^*} {\rm dist}(x,\partial \mathcal{D})>0$. One of the surprising features of infinite-dimensional holomorphy is that the
inclusion $f\in\Hol(\mathcal{D},Y)$ does not imply that $f$ is bounded on all subsets $\mathcal{D}^*$ strictly inside $\DD$ (see \cite{Har, R-S1, R-S97, R-S98}).

We proceed with some general notions  concerning the continuity of an arbitrary family of mappings $\left\{f_t\right\}_{t\ge0}\subset\Hol(\DD, Y)$.

\begin{defin}\label{def-T-cont}
The family $\left\{f_t\right\}_{t\ge0}\subset \Hol(\DD, Y)$ is said to be
\begin{itemize}
\item {\bf{uniformly jointly continuous}} (UJC, for short) if for every point $(t_0,x_0)\in[0,\infty)\times\DD$
there exists a neighborhood $U$ of $x_0$  such that $f_t(x)\rightarrow f_{t_0}(x)$ as $t\to t_0$, uniformly on $U$.

\item {\bf locally uniformly continuous} ($T$-continuous, for short) if for every  $t_0\ge0$ and for every subset $\mathcal{D}^*$ strictly inside~$\DD$, we have  $f_t(x)\rightarrow f_{t_0}(x)$ as $t\to t_0$, uniformly on $\DD^*$.
\end{itemize}
\end{defin}

 As we will see below, these notions are closely connected with the differentiability of a family with respect to the parameter $t$.

A central notion in the theory of dynamical systems is that  of one-parameter semigroups of mappings.

\begin{defin}\label{def-sg-hol}
A family $\mathcal{F} =\left\{F_t\right\}_{t\ge0}\subset\Hol(\DD)$ is called a one-parameter continuous semigroup
(semigroup, for short) on $\mathcal{D}$ if the following properties hold
\begin{itemize}
  \item [(i)]  $F_{t+s}=F_t \circ F_s$ for all $t,s\geq 0$;

  \item [(ii)]  for all $x\in \mathcal{D}$, $\displaystyle \lim_{t\to 0^+}F_t(x) =x.$
\end{itemize}
\end{defin}

Note that, the algebraic structure of semigroups implies that if  $F_t(x)\to x$ as $t \to 0^+$ uniformly on a neighborhood of every point in $\DD$ (respectively, uniformly on every subset strictly inside $\DD$), then it is UJC (respectively, $T$-continuous). One of the deep questions in semigroup theory is whether all semigroups are  differentiable with respect to the parameter $t$. In general, the answer is negative. Reich and Shoikhet (see \cite[Theorems 6.8--6.9]{R-S1}) proved the following criterion for differentiability.
\begin{theorem}\label{th-RS}
  Let $\Ff\subset\Hol(\DD)$ be a semigroup on a bounded domain in a complex Banach space $X$. Then $\Ff$ is $T$-continuous if and only if for each $x\in\DD$  the limit
 \begin{equation}\label{gener}
f (x)= \lim_{t\to 0^{+}}\frac{1}{t} \left[ F_t(x) -x\right]
\end{equation}
exists, uniformly on subsets strictly inside $\mathcal{D}$, and  then $f\in\Hol(\DD,X)$ is bounded on each subset strictly inside $\DD$.
\end{theorem}
The mapping $f$ defined by \eqref{gener} is called
the {\sl (infinitesimal) generator} of the
semigroup~$\mathcal{F}.$ In this case the semigroup $\mathcal{F}$ can be reproduced as the unique solution of the Cauchy problem
\begin{equation}  \label{nS1}
\left\{
\begin{array}{l}
\displaystyle
\frac{\partial u(t,x)}{\partial t}=f(u(t,x)) \vspace{2mm} \\
u(0,x)=x,%
\end{array}%
\right.
\end{equation}%
where we set $u(t,x)=F_{t}(x)$ (see, for example, \cite{E-R-S-19}).

\begin{defin}
Let $\Ff=\{F_t(x)\}_{t\geq 0}$ be a semigroup on a domain $\DD$.
We say that $\Ff$ acts strictly inside $\DD$ if for every subset
$\DD^*$ strictly inside $\DD$ and every $t_0>0$ the set
$\{F_t(x):\ x\in\DD^*,\ t\in[0,t_0] \}$ lies strictly inside
$\DD$.
\end{defin}

In the case where $X$ is finite-dimensional, each semigroup acts strictly inside by compactness, while in infinite-dimensional settings this property should be verified. The next assertion presents a simple sufficient condition for a semigroup to act strictly inside its domain.

\begin{lemma}\label{lem-inside}
Let $\DD$ be a bounded convex domain in a complex Banach space. Then each semigroup
$\mathcal{F}=\{F_t\}_{t\ge0}\subset\Hol(\DD)$ acts strictly inside
$\DD$.
\end{lemma}

\begin{proof}
 Since $\DD$ is a bounded convex domain, there exists a hyperbolic metric $\rho$ on $\DD$. For each $t\ge0$, the semigroup element $F_t$ is $\rho$-nonexpansive in the sense that $\rho(F_t(x),F_t(y))\le\rho(x,y)$ for all $x,y\in\DD$.
	
Take now any subset $\DD^*$ which lies strictly inside $\DD$. Clearly $\DD^*$  is
	contained in a $\rho$-ball centered at $x_0$; denote the radius of this
	ball by $r$. Then for every $x\in\DD^*$ and $t\in[0,t_0]$,
	\begin{eqnarray*}
		\rho (x_0, F_t(x)) &\le& \rho(x_0,F_t(x_0))  +  \rho(F_t(x_0), F_t(x)) \\
		&\le& \rho(x_0,F_t(x_0))  +  \rho(x_0,x)\le   \rho(x_0,F_t(x_0))  +r.
	\end{eqnarray*}
	The first summand in the last sum is finite when $t\le t_0$. Hence
	the set $\{F_t(x):\ x\in\DD^*,\  t\in[0,t_0] \}$ belongs to a $\rho$-ball $\mathcal B$ centered at $x_0$ and having a finite radius, so it lies strictly inside $\DD$.
\end{proof}
Note that the conclusion of Lemma~\ref{lem-inside} is valid under weaker conditions too. Indeed, suppose that $\rho$ is a hyperbolic metric on $\DD$, and a family $\{F_t(x)\}_{t\geq 0}$ is continuous with respect to $t$ and consists of $K$-Lipschitz self-mappings relative to $\rho$ on a domain $\DD$. Then following the same considerations we see that  $\Ff$ acts strictly inside $\DD$.

A point $x_0 \in \DD$ is called a fixed point of a semigroup $\mathcal{F} =\left\{F_t\right\}_{t\ge0}$ if $F_t(x_0)=x_0$ for all $t \geq 0$. It follows from the uniqueness of the solution to the Cauchy problem~\eqref{nS1} that the common fixed point set of
$\mathcal{F}$ coincides with the null point set of its generator~$f$. A fixed point $x_0$ is said to be (globally) {\it{attractive}} if  $F_t(x)\to x_0$ as $t\to\infty$ for all $x\in \DD$. While in the one-dimensional case, for a semigroup of holomorphic self-mappings, which does not consist of automorphisms, the existence of a unique fixed point implies its attractivity, this is no longer true in higher-dimensional spaces, as shown by the simple example: $F_t(x_1,x_2)=(e^{-t}x_1,e^{it}x_2)$. Moreover, in an infinite-dimensional space a semigroup may converge to a unique interior point  $x_0\in\DD$ but not uniformly on any neighborhood of $x_0$ (see, for instance, Example~\ref{examp123} below).

The most familiar class of semigroups consists of semigroups of bounded linear or affine operators. In the general situation, the `linearization problem' is the question whether a given semigroup is biholomorphically equivalent to a semigroup of linear/affine mappings. If $\Ff$ has an interior fixed point $x_0\in\DD$, this problem is equivalent to searching for a biholomorphic mapping $h\in\Hol(\DD,X),\ h(x_0)=0,$ and a linear operator $A$ on $X$ such that
    \begin{equation}\label{konigs}
    F_t(x)=h^{-1}\left(e^{tA}  h(x) \right).
    \end{equation}
    Obviously, if $\Ff$ is generated by a mapping $f\in\Hol(\DD,X)$, then $A=f'(x_0)$.

    While in the one-dimensional case any semigroup is linearizable (see \cite{E-S-book} for a survey of this problem), in the multi-dimensional settings there are  non-linearizable semigroups. A criterion for linearizability of a semigroup acting on a domain in the finite-dimensional space $\C^n$ is given in \cite{B-E-S}. Moreover, it is shown there that if the operator $A=f'(x_0)$ has no resonances, then the semigroup generated by $f$ is linearizable. Regarding the general infinite-dimensional case, a sufficient condition for linearizability involving first terms of Taylor's series of the semigroup generator is given in \cite[Proposition~3.7.5]{E-R-S-04}.

\vspace{3mm}

The main object of study in this paper is that of
{\it{semicocycle}}, which plays
an important role in the theory of dynamical systems (for the
holomorphic one-dimensional case see~\cite{Konig}). Throughout the paper we assume that $\A$ is a complex unital Banach algebra with
the unity $1_A$ such that $\|1_\A\|_\A=1$.

\begin{defin}\label{semicocycle}
Let $\mathcal{F}=\{F_t\}_{t\ge0}\subset\Hol(\DD)$ be a semigroup. The family $\left\{\Gamma_t\right\}_{t\ge0}\subset\Hol(\DD,\A)$ is called a (holomorphic) semicocycle over $\mathcal{F}$ if it satisfies the following:

\begin{itemize}
  \item [(a)] the chain rule:
$\Gamma_t(F_s(x))\Gamma_s(x) = \Gamma_{t+s}(x)$ for all $t,s\ge0$
and $x\in\DD$;

\item [(b)] $\displaystyle \lim_{t\to 0^+}\Gamma_t(x)= 1_\A$ for
every $x \in \DD$.
\end{itemize}
\end{defin}
The simplest example of a semicocycle over $\mathcal{F}$ is given by the Fr\'ech{e}t derivatives of the semigroup: $\Gamma_t(x)={F_t}'(x)$. It can be easily seen that if the semigroup $\Ff$ is generated by a mapping $f$, this semicocycle satisfies the differential equation $\displaystyle \frac{d v(t,x)}{d t} =f'(F_t(x))v(t,x)$. It turns out that this equation is a very special case of a wide class of nonautonomous dynamical systems  solutions of which  are semicocycles. More precisely,
\begin{theorem}[Theorem~4.2 in \cite{EJK-multi}]\label{th-sol-cauchy}
	Let $B\in \Hol(\mathcal{D},\A)$, and let $\mathcal{F}= \{F_t\}_{t\ge0}\subset \Hol(\mathcal{D})$ be a semigroup. The evolution problem
	\begin{equation}\label{cauchy}
	\left\{
	\begin{array}{l}
	\displaystyle \frac{d v(t,x)}{d t} =B(F_t(x))v(t,x) \vspace{2mm} \\
	v(0,x)=1_\A.
	\end{array}%
	\right.
	\end{equation}
	has the unique solution $u(t,x),\ t\ge0,\ x\in \mathcal{D}$, and the family $\{\Gamma_t(x):=u(t,x)\}_{t\ge0}$ is a holomorphic semicocycle over $\Ff$.
\end{theorem}
We call the mapping $B\in \Hol(\mathcal{D},\A)$ the generator of the semicocycle. A~natural question is whether semicocycles in general are solutions of dynamical systems. This question is answered by
\begin{theorem}[Theorems~5.1--5.2 in \cite{EJK-multi}]\label{th-differentiability}
	Let $\DD$ be a bounded domain.
	Let $\{\Gamma_t\}_{t \geq 0}\subset\Hol(\DD,\A)$ be a semicocycle over a $T$-continuous semigroup~$\Ff=\left\{F_t\right\}_{t\ge0}\subset\Hol(\DD)$.
	Then, for each $x\in \DD$, the function $t\mapsto \Gamma_t(x)$  is differentiable on $[0,\infty)$  if and only if $\{\Gamma_t\}_{t \geq 0}$ is UJC. In this case, defining
	\begin{equation}\label{B}
	B(x)=\left.\frac{d}{dt}\Gamma_t(x)\right|_{t=0},
	\end{equation}
	we have ${B\in\Hol(\DD,\A)}$ and $\Gamma_t(x)$ is the unique solution to
	the evolution problem \eqref{cauchy}.
	Moreover, $\left\{\Gamma_t,\right\}_{t\ge0}$ is $T$-continuous if and only if its generator $B$ is bounded on every domain $\mathcal{D}^*$ strictly inside $\DD$.
\end{theorem}
We recall a consequence of the above theorem, which will be of use later (cf. \cite[Theorems 4.4 and 5.3]{EJK-multi}).
\begin{theorem}\label{th-estim1}
	Let $\Ff=\{F_t\}_{t\ge0}\subset \Hol(\DD)$ be a semigroup  such that for some $x_0\in\DD$,  $F_t(x)\to x_0$ as $t\to\infty$, uniformly on subsets strictly inside~$\mathcal{D}$.
	Let $\left\{\Gamma_t\right\}_{t\ge0}\subset \Hol(\DD,\A)$  be an UJC semicocycle over $\Ff$. Then for every set $\DD^*$  strictly inside $\DD$ there exist real $C$ and
	$L$ such that $ \|\Gamma_t(x)\|_\A\le C e^{Lt}$ for all $ x\in\mathcal{D}^*.$
\end{theorem}

Denote the set of all invertible elements of $\A$ by $\A_*$. For any given semicocycle $\left\{\Gamma_t\right\}_{t\ge0}$  we can construct a class of other semicocycles as follows.
\begin{propo}\label{propo-1}
	Let $M\in\Hol(\DD,\A_*)$. Then the family 	$\left\{\widetilde{\Gamma}_t\right\}_{t\ge0}$ defined~by
	\begin{equation}\label{rep1}
	\widetilde{\Gamma}_t(x)=M(F_t(x))^{-1}\Gamma_t(x) M(x)
	\end{equation}
	is a semicocycle over $\mathcal{F}$.
\end{propo}

Two semicocycles $\{\Gamma_t\}_{t\ge0}$ and $\{\widetilde{\Gamma}_t\}_{t\ge0}$ related as in \eqref{rep1} are said to be {\it cohomologous}. It is easy to see that this is an equivalence relation.  Note that a particular case where $\widetilde{\Gamma}_t(x)=1_{\A}$, that is, $\Gamma_t(x)= M(F_t(x))M(x)^{-1}$, is sometimes referred to as a {\it coboundary} (see, for example, \cite{Latus, Katok}).

\section{Linearization problem}\label{sect-lineari}
\setcounter{equation}{0}

The simplest examples of semicocycles are those independent of~$x$,
 \begin{equation}\label{1111}
 \{\Gamma_t(x)=e^{tB_0}\}_{t\ge0}  \quad \mbox{for some}\quad B_0\in\A,
 \end{equation}
so it is natural to ask which semicocycles are cohomolgous to \eqref{1111}, that is can be represented in the form
\begin{equation}\label{rep}
	\Gamma_t(x)=M(F_t(x))^{-1}e^{tB_0}M(x).
	\end{equation}

In view of the analogy between  the concepts of cohomology of semicocycles and conjugacy of semigroups (compare formulas  \eqref{rep}  and \eqref{konigs}), representing a given semicocycle $\{\Gamma_t\}_{t\ge0}$ by~\eqref{rep} will be called {\it linearization} of  $\{\Gamma_t\}_{t\ge0}$ and the mapping $M$ will be called a {\it linearizing mapping} for $\{\Gamma_t\}_{t\ge0}$. Thus, the linearization problem for semicocycles, which is the focus of this paper, is
\begin{itemize}
  \item Under what conditions is a given semicocycle linearizable, that is, can be represented in the form \eqref{rep}?
\end{itemize}
Note that a semicocycle represented by \eqref{rep} is automatically differentiable with respect to $t$, hence uniformly jointly continuous by Theorem~\ref{th-differentiability}.

 We remark that representation  \eqref{rep}, if it exists, is not unique. Since $M$ takes invertible values, we can assume without loss of generality (cf. \cite{EJK}) that the mapping $M$
in \eqref{rep} satisfies the normalization
\begin{equation}\label{norm}
M(x_0)=1_\A
\end{equation}
for an arbitrary chosen $x_0\in\DD$. Denote, as above,  $B(x)=\left.\frac{d}{dt}\Gamma_t(x)\right|_{t=0}$. Assuming (as we will henceforth do) that  $x_0$ is a fixed point of $\Ff$, we get that $B_0$ in~\eqref{rep} is uniquely determined by
\begin{equation}\label{B0}
B_0=B(x_0).
\end{equation}
However, even conditions \eqref{norm}--\eqref{B0} do not guarantee the uniqueness of the linearizing mapping $M$ as the following example shows.

\begin{example}\label{examp-uniq}
	Let $X=\C,\ \DD$ be the open unit disk, $\Ff=\{e^{-t}\cdot\}_{t\ge0}$
	be the linear semigroup. Denote
	\[
	B_0=\left( \begin{array}{cc}
	1 & 0 \\
	0 & 2
	\end{array} \right),\quad
	M(x)=\left( \begin{array}{cc}
	1 & x \\
	0 & 1
	\end{array} \right).
	\]
	By Proposition~\ref{propo-1}, the family
	$\left\{\Gamma_t\right\}_{t\ge0}$ defined by
	$\Gamma_t(x)=M(F_t(x))^{-1}e^{tB_0}M(x)$ forms a semicocycle over
	$\Ff$. Direct calculation shows that actually
	$\Gamma_t(x)=e^{tB_0}$, so that it is linearizable by the constant mapping $1_\A\not=M(x)$.
\end{example}
On the other hand, if either the algebra $\A$ is commutative or
$B_0={\lambda\cdot 1_\A},\ \lambda\in\C,$ then it can be easily seen that the representation
\eqref{rep} is unique. Below in Theorem~\ref{th1} we present a more general condition for uniqueness.

\vspace{2mm}

It turns out that the $T$-continuity of the semicocycle represented by formula~\eqref{rep} can be characterized in terms of the mapping~$M$. To this end we assume in what follows that a semigroup $\Ff\subset\Hol(\DD)$ on a bounded domain~$\DD$ is $T$-continuous, hence is differentiable by Theorem~\ref{th-RS}.

\begin{theorem}\label{propo-gener1}
	Let  $\DD$ be a bounded domain, $M\in\Hol(\DD,\A_*),\ B_0\in\A$ and the family $\left\{\Gamma_t\right\}_{t\ge0}$ be defined by \eqref{rep}.
	
	(i) If both $M$ and $M^{-1}$ are bounded on each subset strictly inside $\DD$, then $\left\{\Gamma_t\right\}_{t\ge0}$ is a $T$-continuous semicocycle over $\mathcal{F}$.
	
	(ii) Assume in addition that $\DD$ is equipped with the hyperbolic metric $\rho$ and $\Ff$ converges to $x_0\in \DD$ as $t\to\infty$ uniformly on sets strictly inside $\DD$. If $\left\{\Gamma_t\right\}_{t\ge0}$ is a $T$-continuous semicocycle over $\mathcal{F}$, then both $M$ and $M^{-1}$ are bounded on each subset strictly inside $\DD$.
\end{theorem}

\begin{proof} (i)
	By Theorem~\ref{th-RS}, the generator $f$ of the semigroup
	$\Ff$ is bounded on each subset strictly inside $\DD$.  Differentiating \eqref{rep},
	we conclude that the semicocycle $\left\{\Gamma_t\right\}_{t\ge0}$
	is generated by the mapping $B\in\Hol(\DD,\A)$ defined by
	\[
	B(x)= M(x)^{-1}\left(B_0 M(x) - M'(x)[f(x)]\right).
	\]
	It follows from the Cauchy inequality (see \cite[Proposition 2.3 ]{R-S1}) that if $M$ is bounded on each subset strictly inside $\DD$ then the same holds for $M'$. Therefore the last formula implies that $B$ also is bounded on each subset strictly inside $\DD$.
	Hence by Theorem~\ref{th-differentiability}, $\left\{\Gamma_t\right\}_{t\ge0}$ is  $T$-continuous.
	
	(ii) Rewrite equality \eqref{rep} in the form
	$$M(x)=e^{-tB_0}M(F_t(x))\Gamma_t(x).$$
	Let $K=\max\left\{ \|M(x_0)\|_\A,\ \|M(x_0)^{-1}\|_\A\right\}$.
	Then there is a neighborhood $U$ of $x_0$ such that
	\[
	\sup_{x\in U} \left\{\|M(x)\|_\A,\ \|M(x)^{-1}\|_\A\right\} \le
	2K.
	\]
	
	Since $\Ff$ converges to $x_0$ as $t\to\infty$ uniformly on sets strictly inside $\DD$, for any domain $\DD^*$ strictly inside $\DD$, there is $t_0>0$ such that $F_t(x) \in U$ whenever $x\in \DD^*$ and $t\ge t_0$. Now by Theorem~\ref{th-estim1}  there are  constants $C$ and $L$ such that $\displaystyle\sup_{x\in \DD^*}\left\|\Gamma_{t_0}(x)\right\|_\A\le Ce^{Lt_0}. $ Hence
	$$\|M(x)\|_\A\le 2K \cdot Ce^{Lt_0} \cdot\|e^{-t_0B_0}\|_\A \quad\mbox{for all}\quad 	x\in \DD^*,$$
so $M$ is bounded on each subset strictly inside $\DD$.

By the proof of Lemma~\ref{lem-inside}, 	the set $\{F_t(x):\ x\in\DD^*,\  t\in[0,t_0] \}$ lies strictly inside $\DD$, so it is contained in  a $\rho$-ball $\mathcal B$ centered at $x_0$. Increasing, if needed, the $\rho$-radius of $\mathcal B$, we get $U\subset\mathcal{B}$. For this  $\rho$-ball $\mathcal B$  there is $t_1\ge t_0>0$  such that  $F_t(x) \in U$ whenever $x\in \mathcal B$ and $t\ge t_1$.
	
	Now, by Theorem~3.2 in \cite{EJK-multi}, the inverse values ${\Gamma_t(x)}^{-1}$ exist. Moreover, since the semicocycle is $T$-continuous, one can find a natural number $n$ such that $\|\Gamma_{t/n}(x)- 1_\A\|_{\A}<\frac12$ for all $x\in\mathcal B$ and $0\le t\le t_1$. Therefore
	\[
	\left\|\Gamma_{t/n}(x)^{-1}\right \|_\A \le \frac1{1- \|\Gamma_{t/n}(x)-1_\A\|_\A} < 2.
	\]
	Since the hyperbolic ball $\mathcal B$ is centered at the fixed point $x_0$, it is $\Ff$-invariant. Applying  the chain rule, we conclude that $\left\|\Gamma_{t_1}(x)^{-1}\right \|_\A\le 2^n.$ Therefore for 	every $x\in \DD^*\subset\mathcal B$ we have
	$$\|M(x)^{-1}\|_\A= 	\|\Gamma_{t_1}(x)^{-1}M(F_{t_1}(x))^{-1}e^{t_1B_0}\|_\A 	\le 2^n\cdot2K\cdot \|e^{t_1B_0}\|_\A, $$
that is, $M^{-1}$ is bounded on each subset strictly inside $\DD$.
\end{proof}

In addition, for a given semigroup $\Ff$ on a domain in a Banach space $X$, semicocycles over $\Ff$ can be used to construct new semigroups on spaces larger than $X$. This construction leads to a direct  connection between linearization of semigroups and linearization of semicocycles. The following fact can be verified by a straightforward calculation (for its first part, see Proposition~3.1 in \cite{EJK-multi}).

\begin{propo}\label{prop-konigs}
        Let $X$ and $Y$ be complex Banach spaces and $\A=L(Y)$. Assume that $\Ff=\{F_t\}_{t\ge0}\subset \Hol(\DD)$ is a semigroup on a domain $\DD\subset X$ and $\Gamma_t:\R^+\to\Hol(\DD, \A)$. The family $\left\{\Gamma_t\right\}_{t\ge0}$ is a semicocycle over $\Ff$ if and only if the family ${\widetilde \Ff} =\left\{\widetilde {F}_t\right\}_{t\ge0}$ defined by  $\widetilde{F}_t(x,y)= \left(F_t (x), \Gamma_t (x) y \right)$ forms a  semigroup on the domain  $\DD\times Y$.

        If, in addition, the semigroup $\mathcal{F}$ is linearizable by a biholomorphic mapping $h\in\Hol(\DD,X)$ in the sense of \eqref{konigs} and  $\left\{\Gamma_t\right\}_{t\ge0}$ is represented in the form \eqref{rep}, then the  semigroup $\widetilde \Ff$  is linearizable by the mapping $\widetilde h$ defined by $\widetilde{h}(x,y)=\left(h(x),M(x)y\right)$.
\end{propo}

The semigroup $\widetilde{\Ff}$ was studied in \cite{E-2011} as an extension operator for semigroups of holomorphic mappings. We see that such operators necessarily involve semicocycles. Note also that the extended semigroup $\mathcal{\widetilde F}$ is sometimes referred to as a linear skew-product flow; see, for example \cite{Latus}.

\section{Main results}\label{Section-linearization}
\setcounter{equation}{0}

We start this section with a simple  sufficient condition for a semicocycle $\left\{\Gamma_t\right\}_{t\ge0}$ to be linearizable.

\begin{propo}\label{cr}
    Let a semigroup $\Ff=\left\{F_t\right\}_{t\ge0}\subset\Hol(\DD)$ have
    an attractive fixed point $x_0 \in \DD$ and
$\left\{\Gamma_t\right\}_{t\ge0}\subset\Hol(\DD,\A)$ be a
    semicocycle over~$\Ff$. Denote $B_0=\left.\frac{d}{dt}\Gamma_t(x_0)\right|_{t=0}.$ If the limit
    \begin{equation}\label{limit1}
    \lim_{t\rightarrow \infty }e^{-tB_0}\Gamma_t(x)=:M(x)
    \end{equation}
    exists locally uniformly, then the semicocycle $\left\{\Gamma_t\right\}_{t\ge0}$ can be represented in the form \eqref{rep}.
\end{propo}
This result is a very special case of Theorem~\ref{th-present-Gamma1} below.

\begin{example}\label{ex-2}
Let $X=c_0$ be the space of all sequences converging to zero. Let $\Gamma_t$ be defined by $ \Gamma_t(x)=\exp\left[ \sum\limits_{k=1}^\infty \left(2x_k\right)^k \left( 1- e^{-tk} \right)\right].$ It was shown in \cite[Example~3.1]{EJK-multi} that the family $\{\Gamma_t\}_{t\ge0}$ is a semicocycle over the linear semigroup $\{e^{-t}\cdot\}_{t\ge0}$. It is easy to see that $B_0=\left.\frac{d}{dt}\Gamma_t(0)\right|_{t=0}=0$ and $M(x)=\lim\limits_{t\to\infty} \Gamma_t(x) =\exp\left[ \sum\limits_{k=1}^\infty \left(2x_k\right)^k \right]$, uniformly on subsets strictly inside the unit ball of $c_0$. So, by Proposition~\ref{cr}, $\Gamma_t(x)=M(e^{-t}x)^{-1}M(x)$. This representation can also be verified directly.
\end{example}

   Next we present a condition on the generator $B$ under which the existence of a linearization is ensured.

\begin{propo}\label{sc}
    Let a semigroup $\mathcal{F}=\left\{F_t\right\}_{t\ge0}\subset\Hol(\DD)$ have an attractive fixed point $x_0 \in \DD$. Let $\left\{\Gamma_t\right\}_{t\ge0}$
    be a semicocycle over $\Ff$ generated by a mapping $B\in\Hol(\DD,\A)$ and $B_0=B(x_0)$.     If for every subset $\DD^*$ strictly inside $\DD$ there is a constant $L$ such that the integral
\begin{equation}\label{cc}
   L(x):= \int_0^\infty \left\|\exp(-tB_0)B(F_t(x))\exp(tB_0)-B_0\right\|_\A  dt
\end{equation}
converges uniformly with respect to $x\in\DD^*$ and $L(x)\le L$, then the limit \eqref{limit1} exists uniformly on each subset strictly inside $\DD$. Hence
$\left\{\Gamma_t\right\}_{t\ge0}$ admits a representation \eqref{rep}.
\end{propo}

\begin{proof}
Let $\DD^*$ be a subset strictly inside $\DD$. For any $x\in \DD^*$, denote $v(t)=\exp(-tB_0) \Gamma_t(x)$. Then
\begin{equation*}\label{for_representation}
v'(t)=\left(\exp(-tB_0)B(F_t(x))\exp(tB_0)-B_0\right)v(t),
\end{equation*}
so $v$ is the unique solution of the evolution problem
\begin{equation*}
\left\{
\begin{array}{l}
\displaystyle v'(t) =\widetilde{B}(t)v(t) \vspace{2mm} \\
v(0)=1_\A,
\end{array}%
\right.
\end{equation*}
where $\widetilde{B}(t)=\exp(-tB_0)B(F_t(x))\exp(tB_0)- B_0 $.
Since the integral in \eqref{cc} converges uniformly on $\DD^*$, it follows from a result by Almkvist \cite{AG1964} that
\begin{equation}\label{estim-u2}
\|v(t)\|_\A\le \exp \left(\int_0^t \|\widetilde{B}(s)\|_\A ds
\right)<e^L
\end{equation}
(independently of $x$); see also \cite[Theorem 4.1]{EJK-multi}. In addition,
\begin{equation*}\label{integral-form2}
v(t) =  1_\A + \int_0^t \widetilde{B}(s)v(s) ds.
\end{equation*}
(see, for example, \cite{M-S, Kr}). Using  \eqref{estim-u2}, we conclude that for any $\varepsilon>0$
there exists $t_\varepsilon>0$ such that
\begin{equation*}\label{integral-form-cauchy}
\|v(t_2)-v(t_1)\|_\A= \left\|\int_{t_1}^{t_2}\widetilde{B}(s)v(s)
ds \right\|_\A \leq  e^L\int_{t_1}^{t_2}\|\widetilde{B}(s)\|_\A ds
< \varepsilon
\end{equation*}
for all $t_2>t_1>t_\varepsilon$. Consequently, the limit
\[
\lim\limits_{t \to \infty} v(t)= \lim_{t\rightarrow \infty
}e^{-tB_0}\Gamma_t(x)
\]
exists uniformly on $\DD^*$ by the Cauchy criterion of convergence. Hence
$\left\{\Gamma_t\right\}_{t\ge0}$ admits a representation
\eqref{rep} by Proposition~\ref{cr}.
\end{proof}

While the conditions given by Propositions \ref{cr} and \ref{sc} are useful, they are sufficient but {\it{not}} necessary for the
existence of a representation (\ref{rep}). This fact is shown by the following example.

\begin{example}\label{ex1}
Let  $\DD$ be the open unit disk in the complex plane $\C$ and $\A=\C^{2\times2}$. Consider the linear semigroup $\mathcal{F}=\{F_t=e^{-t}\cdot \} _{t\ge0}$ and denote
    $$B_0=\left( \begin{array}{cc}
    3 & 0 \\
    0 & 1
    \end{array} \right),\;\;\;M(x)=\left( \begin{array}{cc}
    1 & x \\
    x & 1
    \end{array} \right).$$
   Define
\begin{eqnarray*}
    \Gamma_t(x)&:=&M(F_t(x))^{-1}e^{tB_0}M(x) \\
    &=&\frac{1}{1-e^{-2t}x^2} \left( \begin{array}{cc} e^{3t}-x^2 & x(e^{3t}-1) \\
    x(e^t-e^{2t}) & e^t-x^2 e^{2t}
\end{array} \right).
\end{eqnarray*}
By construction, $\Gamma_t$ has the representation \eqref{rep}.
However, the limit
$$\lim_{t \rightarrow \infty}e^{-tB_0}\Gamma_t(x)=\lim_{t \rightarrow \infty}\frac{1}{1-e^{-2t}x^2}
\left( \begin{array}{cc}
1-e^{-3t}x^2 & x(1-e^{-3t}) \\
x(1-e^{t}) & 1-x^2 e^{t}
\end{array} \right) $$
does not exist. Therefore, as it follows from the proof of Proposition~\ref{sc}, the integral~\eqref{cc} diverges.
\end{example}

The convergence of the  improper integral in~\eqref{cc} depends on
two independent objects: the semigroup $\Ff$ and the semicocycle
generator~$B$. Intuitively, if the rate of convergence of $B(F_t(x))$ to $B_0$
is sufficiently fast, the integral~\eqref{cc} will
converge. Therefore, one can expect that
some condition on the values $A=f'(x_0)$ and $B_0=B(x_0)$ may ensure
the convergence of that integral.  To formulate such condition, we recall that for an element $a$ of a Banach algebra, the upper
and lower exponential (Lyapunov) indices of $a$  (see, for example, \cite{Kr} or \cite{D-Sch1958}) by
\begin{equation}\label{kappa+}
  \kappa_+(a):=   \limsup_{t\rightarrow \infty}\frac{1}{t} \log\|e^{ta}\|
\end{equation}
and
\begin{equation}\label{kappa-}
  \kappa_-(a):= -   \limsup_{t\rightarrow \infty}\frac{1}{t} \log\|e^{-ta}\|.
\end{equation}
We now define the {\it characteristic ratio} by
\begin{equation}\label{ell}
\ell:= \frac{\kappa_+(B_0)-\kappa_-(B_0)} {|\kappa_+(A)|}\,, \quad \mbox{where }\ A=f'(x_0).
\end{equation}

In the sequel we consider the following condition on a semigroup:
\begin{itemize}
  \item [(*)] The semigroup $\mathcal{F}\subset\Hol(\DD)$ converges to a point $x_0\in\DD$ uniformly on subsets strictly inside $\DD$ and is generated by a mapping $f$  such that
\begin{equation}\label{est-num-0}
\kappa_+(A)<0,\quad \mbox{where }\ A=f'(x_0);
\end{equation}
\end{itemize}

In fact, condition (*) provides a local estimate of the rate of convergence of $F_t(x)$ to $x_0$ as $t\to\infty$ (see, for example, \cite{D-S}). Moreover, in the case where $\DD$ is the open unit ball in $X$, inequality \eqref{est-num-0} by itself implies the uniform convergence on every subset strictly inside $\DD$; see \cite{E-R-S-04}.

\begin{theorem}\label{sc1}
Let $\DD$ be a bounded domain equipped with the hyperbolic metric. Let a semigroup $\Ff\subset \Hol(\DD)$ satisfy condition (*) and $\left\{\Gamma_t\right\}_{t\ge0}\subset\Hol(\DD, \A)$ be a $T$-continuous semicocycle over $\Ff$ generated by $B\in\Hol(\DD,\A)$ with $B_0=B(x_0)$.\\
If $
\displaystyle\lim_{x\to x_0}\frac{\|B(x)-B_0 \|_\A} {\|x-x_0\|_{_X}^{\ell}} = 0,
$
where the characteristic ratio $\ell$ is defined in \eqref{ell}, then $\left\{\Gamma_t\right\}_{t\ge0}$ can be represented in the form~\eqref{rep}, where $M$ is defined by limit \eqref{limit1}.
\end{theorem}

\begin{proof}
It follows from the assumptions that for some neighborhood $U$ of~$x_0$ strictly inside $\DD$ there are $L>0$ and $\delta \in (0,1]$ such that
 \begin{equation}\label{est-B}
\|B(\tilde x)-B_0\|_\A \leq L \|\tilde x -x_0\|_X^{\ell+\delta}\quad\mbox{ for all }\quad  \tilde x \in U.
\end{equation}
For any $\epsilon>0$, one can find a constant $C>0$ and a neighborhood $U^*$ of~$x_0,\  U^*\subset U$, such that
\[
\|F_t(x)-x_0\|_X \le Ce^{t(\kappa_+(A) +\epsilon)}
\quad \mbox{for all}\quad x\in U^*,\ t\ge0
\]
(see, for example, \cite{D-S}).

Let $\DD^*$ be a subset strictly inside $\DD$. Due to condition (*), $\Ff$ converges uniformly on $\DD^*$. So there is a  positive number  $t^*$ such that $F_t(x)\in U^*$ whenever $t\ge t^*$ and $x\in\DD^*$.
Therefore for all $t>0$ and  $x\in\DD^*$ we have
\[
\|F_{t+t^*}(x)-x_0\|_X \le Ce^{t(\kappa_+(A) +\epsilon)}
.
\]
Substituting $\tilde x=F_{t+t^*}(x)$ in \eqref{est-B},  we conclude that for all $x\in\DD^*$ the last inequality implies
\begin{eqnarray}\label{estim-B}
\left\| B(F_{t+t^*}(x))-B_0 \right\|_\A 
\le LC^{\ell+\delta}e^{t(\kappa_+(A) +\epsilon)(\ell+\delta)}.
\end{eqnarray}
We now use this   to estimate the integral~\eqref{cc}. By Lemma~\ref{lem-inside}, the set
$\Omega:=\{F_t(x):\ x\in\DD^*,\ t\in[0,t^*] \}$ lies strictly
inside $\DD$. Then, by Theorem~\ref{th-differentiability} the  mapping $B$ is bounded on
$\Omega$, and hence the integrals $$ \int_0^{t^*}
\left\|\exp(-B_0 t)B(F_t(x))\exp(B_0
   t)-B_0\right\|_\A  dt$$ are uniformly bounded on $\DD^*$.

To proceed, recall (see, for
example, \cite{Kr} or \cite{D-Sch1958}) that
there are constants $L_1=L_1(\epsilon)$ and $L_2=L_2(\epsilon)$
such that
    \[
\left\|e^{tB_0}\right\|_\A \le L_1 e^{t(\kappa_+(B_0) +\epsilon)}
\quad\mbox{and}\quad \left\|e^{-tB_0}\right\|_\A \le L_2
e^{t(\epsilon-\kappa_-(B_0))}.
    \]

Therefore
\begin{eqnarray*}
&&\|\exp(-B_0 t)B(F_t(x))\exp(B_0 t)-B_0\|_\A \\
&=& \|\exp(-B_0 t)[B(F_t(x))-B_0]\exp(B_0 t)\|_\A \\
&\leq& \|\exp(-B_0 t)\|_\A \cdot \|B(F_t(x))-B_0\|_\A \cdot \|\exp(B_0 t)\|_\A\\
&\le& L_1L_2\exp\left[t(\kappa_+(B_0)-\kappa_-(B_0)+2\epsilon)\right] \cdot
\|B(F_t(x))-B_0\|_\A.
\end{eqnarray*}

Thus \eqref{estim-B} implies that for all $t\ge t^*$
\begin{eqnarray*}
\|\exp(-B_0 t)B(F_t(x))\exp(B_0 t)-B_0\|_\A   \leq \widetilde{C}
e^{t(\delta\kappa_+(A)+(2+\ell+\delta)\epsilon)},
\end{eqnarray*}
where $\widetilde{C}=LL_1L_2C^{\ell+\delta}.$

Since $\delta \kappa_+(A)<0$ and $\epsilon$ can be chosen arbitrary close to zero, we conclude that the
integrals $$\int_{t^*}^\infty \left\|\exp(-B_0 t)B(F_t(x))\exp(B_0
   t)-B_0\right\|_\A  dt$$ converge uniformly on $\DD^*$ and are uniformly bounded. Thus our
conclusion follows by Proposition~\ref{sc}.
\end{proof}

\begin{corol}
If under conditions of Theorem~\ref{sc1}, $\ell<1$, that is,
$$\kappa_+(B_0)-\kappa_-(B_0)+\kappa_+(A)<0,$$
then $\left\{\Gamma_t\right\}_{t\ge0}$ can be represented in the form~\eqref{rep}.
\end{corol}

In particular, applying the case $B_0=B(x_0)=0$ we obtain immediately
\begin{corol}
Let $\DD$ be a bounded domain equipped with the hyperbolic metric. Let a semigroup $\Ff\subset \Hol(\DD)$ satisfy condition (*) and let a $T$-continuous semicocycle $\left\{\Gamma_t\right\}_{t\ge0}\subset\Hol(\DD, \A)$ over $\Ff$ satisfy $\left. \frac{d\Gamma_t(x_0)}{dt}\right|_{t=0}=0 $. Then $\left\{\Gamma_t\right\}_{t\ge0}$ is a coboundary, in particular, is linearizable.
\end{corol}

\medskip

We now present the main result of this paper which gives a necessary and sufficient condition for a semicocycle to be linearizable. Proposition~\ref{cr} above is a particular case of this assertion.

\begin{theorem}\label{th-present-Gamma1}
Assume that a semigroup $\mathcal{F}\subset\Hol(\DD)$ converges to a point $x_0\in\DD$ and $\left\{\Gamma_t\right\}_{t\ge0}\subset\Hol(\DD,\A)$ is a semicocycle over~$\Ff$. Denote $B_0=\left.\frac{d}{dt}\Gamma_t(x_0)\right|_{t=0}.$
    Then this semicocycle is linearizable if and only if there exists a
    holomorphic mapping $N:\DD\to\A$ with $N(x_0)=1_\A$ such that the limit
    \begin{equation}\label{limit}
\lim_{t\rightarrow \infty }e^{-tB_0}N(F_t(x))\Gamma_t(x) =:M(x)
    \end{equation}
    exists, uniformly on each subset strictly inside $\DD$. In this case the mapping $M$ defined by \eqref{limit} linearizes $\left\{\Gamma_t\right\}_{t\ge0}$.

    Assuming, moreover, that $\Ff$ satisfies condition (*), we can choose the mapping $N$ giving \eqref{limit} to be a polynomial mapping of degree not exceeding $\ell$ defined by \eqref{ell}.
\end{theorem}

\begin{proof}
    Assume that a representation of the form \eqref{rep} exists, and set $N(x)=M(x)$. Then
    \begin{eqnarray*}
    &&e^{-tB_0}N(F_t(x))\Gamma_t(x)=
    e^{-tB_0}M(F_t(x))\Gamma_t(x) \\
    &=&  e^{-tB_0}M(F_t(x))M(F_t(x))^{-1}e^{tB_0}M(x)=
    M(x).
    \end{eqnarray*}

Conversely, assume that the limit \eqref{limit} exists, uniformly
on each subset  strictly inside $\DD$, and defines the mapping
$M$. Then, we have, by the uniform convergence, that $M$ is
holomorphic and satisfies
\begin{eqnarray*}
M(F_t(x))\Gamma_t(x)&=&\lim_{s\rightarrow \infty }e^{-sB_0}N(F_{s+t}(x))\Gamma_s(F_t(x))\Gamma_t(x)\\
&=&\lim_{s\rightarrow \infty }e^{-sB_0}N(F_{s+t}(x))\Gamma_{s+t}(x) \\
&=& e^{tB_0}\lim_{s\rightarrow \infty
}e^{-(s+t)B_0}N(F_{s+t}(x))\Gamma_{s+t}(x)=e^{tB_0}M(x),
\end{eqnarray*}
    so we have
\begin{equation}\label{e}M(F_t(x))
    \Gamma_t(x)=e^{tB_0}M(x).
\end{equation}
    We now show that $M(x)$ is invertible for all $x\in\DD$.
    Indeed,
\[
M(x_0)=\lim_{s\rightarrow \infty }e^{-sB_0}N(x_0)\Gamma_s(x_0)
    =\lim_{s\rightarrow \infty }e^{-sB_0}e^{sB_0}=1_\A.
\]
    Fixing $x$, we have
    $$\lim_{t\rightarrow \infty} M(F_t(x))=M(x_0)=1_\A,$$
    which implies that there exists $t_0$ so that $M(F_t(x))$ is
    invertible for $t\ge t_0$.
    From \eqref{e} we have
    $$M(x)=e^{-tB_0}M(F_t(x))\Gamma_t(x),$$
    and for $t\ge t_0$ the right-hand side is invertible, hence $M(x)$ is invertible.

    Therefore \eqref{e} can be re-written as
    $$\Gamma_t(x)=M(F_t(x))^{-1}e^{tB_0}M(x),$$
    and we have proved the existence of the required representation.

    Assume now that the semigroup $\Ff$ satisfies condition (*).
    Since $M$ is holomorphic, it is  bounded on some neighborhood $U$ of~$x_0$, strictly inside~$\DD$.
    In addition, it can be represented in a neighborhood of $x_0$ by a Taylor series, that is, a series of homogenous polynomials in $x-x_0$.
    We can write $M=P+M_1$, where $P$ consists of all homogenous polynomials of $x-x_0$
    of degree not greater than $\ell$ and $M_1$ contains all others. Note that the mapping $M_1=M-P$ is bounded on $U$. Therefore, there are $0<\delta\le 1$ and $L_1=L_1(U)$ such that
    $\displaystyle
    \|M_1(\tilde x)\|_\A\le L_1{\|\tilde x-x_0\|_X}^{\ell+\delta}$ whenever $x \in U$.

    For any $\epsilon>0$, one can find a constant $C>0$ and a neighborhood $U^*$ of~$x_0,\  U^*\subset U$, such that
    \[
    \|F_t(x)-x_0\|_X \le Ce^{t(\kappa_+(A) +\epsilon)}
    \quad \mbox{for all}\quad x\in U,\ t\ge0
    \]
    (see, for example, \cite{D-S}).
    Furthermore, for every subset $\DD^*$ strictly inside $\DD$ there is $t^*$ such that $F_t(x)\in U^*\subset U$ for all $t>t^*$ and $x\in\DD^*$.

    Therefore we have for all $x\in\DD^*$ and $t>0$,
    \[
    \|M_1\left(F_{t+t^*}(x)-x_0\right)\|_\A \le L_1{\|F_{t+t^*}(x)-x_0\|_X}^{\ell+\delta}
    \le L_1 C^{\ell+\delta}e^{t(\kappa_+(A) +\epsilon)(\ell+\delta)}\,.
    \]

    By \eqref{kappa-}, there is a constant     $L_2=L_2(\epsilon)$ such that $  \left\|e^{-tB_0}\right\|_\A \le L_2 e^{t(\epsilon-\kappa_-(B_0))}.$
    In addition, Theorem~\ref{th-estim1} asserts the existence of  $L_3=L_3(r,\epsilon)$  such that $ \left\|\Gamma_t(x)\right\|_\A \le L_3 e^{t(\kappa_+(B_0)+\epsilon)}$ for all $x\in\DD^*$.

    Combining all of these estimates, we conclude that for $t>t^*$,
    \begin{eqnarray*}
        && \left\|e^{-tB_0}M_1(F_t(x)-x_0)\Gamma_t(x)\right\|_\A  \le \\
        &\le& L_2
        e^{t(\epsilon-\kappa_-(B_0))} \cdot L_1 C^{{\ell+\delta}}e^{(t-t^*)(\kappa_+(A) +\epsilon)(\ell+\delta)} \cdot L_3  e^{t(\kappa_+(B_0)+\epsilon)} \\
        &=&\widetilde{C} e^{t((2+\ell+\delta)\epsilon + \delta \kappa_+(A) )}
        \,.
    \end{eqnarray*}

    Since $\delta \kappa_+(A)<0$, the last norm tends to zero as $t\to\infty$ for
    $\epsilon$ small enough. Thus
    \begin{eqnarray*}
        \lim_{t\rightarrow \infty }e^{-tB_0}P(F_t(x))\Gamma_t(x) &=&
        \lim_{t\rightarrow \infty }e^{-tB_0}\bigl(M(F_t(x)) -
        M_1(F_t(x)-x_0)\bigr)\Gamma_t(x) \\
        &=& \lim_{t\rightarrow \infty }e^{-tB_0}M(F_t(x))\Gamma_t(x)=M(x),
    \end{eqnarray*}
    so the desired result follows.
\end{proof}

To illustrate the last assertion of this theorem, let us return to Example~\ref{ex1}. We already know that the limit \eqref{limit} does not exist for $N$ being a polynomial of degree $0$ but it exists for $N=M$, which is a polynomial of degree $\ell=2\ge1$.

Furthermore, by this theorem if condition (*) holds, then $N$ in \eqref{limit} can be chosen to be a polynomial. It turns out that if condition (*) does not hold, it might happen that the linearizing mapping cannot be found by using uniform limits including polynomials.

\begin{example}\label{examp123}
Let  $X=c_{0}$ equipped  with the $\sup$-norm. Consider the semigroup
$\mathcal{F}=\{F_t\}_{t\ge0}$ on the open unit ball $\DD$ of $X$ defined
by
\[
F_t(x_1,x_2,x_3,...)=\left(e^{-t}x_1,e^{-t/2}x_2,e^{-t/3}x_3,... \right).
\]
Obviously, the convergence $F_t(x)\rightarrow 0$ is not uniform on
any ball centered at zero.
Denote $M(x)=\exp\left[
\sum_{k=1}^\infty \left(2x_k\right)^k \right]$. The semicocycle $\{\Gamma_t\}_{t\ge0}\subset\Hol(\DD,\C)$ defined
by $\Gamma_t(x)=M(F_t(x))^{-1}M(x)$ is linearizable by
construction. We wish to show that the mapping $M$ cannot be
obtained by limit~\eqref{limit} with  a polynomial mapping $N$.

Suppose on contrary that  $M(x)=\lim\limits_{t \to \infty}N(F_t(x))M(F_t(x))^{-1}M(x)$ for some polynomial $N$ and then
\begin{equation}\label{P-lin}
\lim\limits_{t \to \infty}N(F_t(x))M(F_t(x))^{-1}=1
\end{equation}
uniformly on every ball of radius less than $1$.
Since each polynomial is bounded on the unit ball of $X$, there
exists a constant $C>0$ such that
\begin{equation*}\label{P-bound}
|N(F_t(x))-1|\le C\|F_t(x)\|_X\le C\|x\|_X.
\end{equation*}
Take  the points $x^{(n)}$  defined by
\[
x^{(n)}  := \Bigl(   \underbrace{ \frac12,\ldots,\frac12}_{n\ \mbox{\small\rm times}},0,0,\ldots\Bigr).
\]
For these points we have $\|x^{(n)}\|_X=\frac{1}{2}$, and hence
$\left|N\left(F_t(x^{(n)})\right)-1\right|\leq \frac {1}{2}C$. In
addition, $ M\left(F_t(x^{(n)})\right)^{-1} =\exp\left[
-e^{-t}\sum_{k=1}^n \left(2x^{(n)}_k\right)^k
\right]=e^{-ne^{-t}}. $ Then
\[
\left|N\left(F_t(x^{(n)})\right)-1\right|\cdot
M\left(F_t(x^{(n)})\right)^{-1} \leq \frac {1}{2}C e^{-ne^{-t}} .
\]
It now follows by the triangle inequality that
\begin{equation*}\label{PM-triangle}
\left|M\left(F_t(x^{(n)})\right)^{-1}-1\right|-\left|N\left(F_t(x^{(n)})\right)M\left(F_t(x^{(n)})\right)^{-1}-1\right|\leq
\frac {1}{2}C e^{-ne^{-t}}.
\end{equation*}
Thus for every $t>0$, we have
\begin{eqnarray*}\label{PM-triangle}
&&\left|N\left(F_t(x^{(n)})\right)M\left(F_t(x^{(n)})\right)^{-1}-1\right|\geq
\left|e^{-ne^{-t}}-1\right| -\frac {1}{2}C e^{-ne^{-t} } \\
&&=1-e^{-ne^{-t}}\left(1+\frac {1}{2}C \right),
\end{eqnarray*}
which is larger than $\frac{1}{2}$ for $n$ sufficiently large. Consequently, the limit in \eqref{P-lin} is not uniform on the ball of radius $\frac {1}{2}$. Thus $M$ cannot be represented by the uniform limit \eqref{limit} with a polynomial
mapping $N$.
\end{example}

\medskip

Note now that the difference $\kappa_+(B_0)-\kappa_-(B_0)$ that appears in Theorems~\ref{sc1}--\ref{th-present-Gamma1} is the `horizontal width' of the spectrum of $B_0\in\A$ and is also equal to $\max\{\Re(\lambda-\lambda'):\ \lambda,\lambda'\in \sigma(B_0)\}$. Therefore it is natural to ask about a more refined condition for linearizability in terms of the spectrum itself. This was investigated in \cite{EJK}, in the case $X=\C$, where a sharp condition for linearizability was obtained in terms of the spectrum of $B_0$. The power-series method employed in \cite{EJK} can be extended to the more general framework considered here only under the restriction that the linear part of the semigroup generator at $x_0$ is a scalar operator.

    \begin{theorem}\label{th1}
Let a semigroup $\Ff$ satisfy condition (*) with ${f'(x_0)=-\omega\Id_X}$. Let $\{\Gamma_t\}_{t\geq 0}$ be a semicocycle over $\Ff$ generated by
$B\in\Hol(\DD,\A)$. If
    \[
k\omega\not\in   \sigma(B_0)-\sigma(B_0):=  \{\lambda-\lambda':\ \lambda,\lambda'\in \sigma(B_0)\}       \quad\mbox{for all}\quad k\in\N,
    \]
then the semicocycle $\{\Gamma_t\}_{t\ge0}$ is linearizable, that
is, it can be represented in the form
  \[
\Gamma_t(z)=M(F_t(z))^{-1}e^{tB_0}M(z),
  \]
and the  mapping $M\in\Hol(\DD,\A_*)$ is unique.
    \end{theorem}
The proof of this theorem is very similar to the proof of Theorem~3.3 in~\cite{EJK} and uses the linearization model for semigroups investigated in \cite{E-R-S-04}.

\medskip

When the operator $f'(x_0)$ is not scalar, obtaining sharp generic conditions for linearizability of semicocycles remains an open question.

\bigskip


\end{document}